\documentclass[11pt,reqno]{amsart}
\usepackage{diagrams}
\usepackage{amssymb}

\theoremstyle{definition}
\newtheorem{lm}{Lemma}[section]
\newtheorem{thm}[lm]{Theorem}

\newtheorem{lem}[lm]{Lemma}
\newtheorem{prop}[lm]{Proposition}

\newtheorem{defi}[lm]{Definition}

\newtheorem{rem}[lm]{Remark}
\newtheorem{note}[lm]{Notation}

\theoremstyle{remark}

\newcommand{\mc}[1]{\mathcal{#1}}
\newcommand{\mr}[1]{\mathrm{#1}}
\newcommand{\ov}[1]{\overline{#1}}
\newcommand{\ul}[1]{\underline{#1}}
\newcommand{\mf}[1]{\mathfrak{#1}}

\DeclareMathOperator{\msp}{\mathrm{Spec}}

\DeclareMathOperator{\Hc}{\mathcal{H}om}

\DeclareMathOperator{\tx}{\tilde{Y}}

\DeclareMathOperator{\Spec}{\mathrm{Spec}}
\DeclareMathOperator{\pr}{\mathrm{pr}}

\DeclareMathOperator{\mo}{\mathcal{O}}

\DeclareMathOperator{\red}{\mathrm{red}}

\begin{document}

\title[Existence of semistable vector bundles with fixed determinants]{Existence of semistable vector bundles with fixed determinants}

\author[I. Kaur]{Inder Kaur}

\address{ Instituto Nacional de Matem\'{a}tica Pura e Aplicada, Estr. Dona Castorina, 110 - Jardim Bot\^{a}nico, Rio de Janeiro - RJ, 22460-320,  Brazil
}

\email{inder@impa.br}

\subjclass[2010]{Primary $14$D$20$, $14$J$60$, Secondary $14$L$24$, $14$D$22$}


\date{\today}

\begin{abstract}
Let $R$ be an excellent Henselian discrete valuation ring with algebraically closed residue field $k$ of any characteristic. Fix integers $r,d$ with $r\ge2$.
Let $X_R$ be a regular fibred surface over $\msp(R)$ with special fibre denoted $X_k$, a  generalised tree-like curve of genus $g \ge 2$. 
Let $\mc{L}_R$ be a line bundle on $X_R$ of degree $d$ such that the degree of the restriction of $\mc{L}_R$ on the rational components of $X_k$ is a multiple of $r$.
In this article we prove the existence of a rank $r$ locally free sheaf on $X_R$ of determinant $\mc{L}_R$ such that it is semistable on the fibres.
\end{abstract}

\maketitle

\section{Introduction}
It is well-known that the moduli space of semistable locally free sheaves of fixed rank and degree on a surface may be empty (see for example {\cite[Section $5.3$]{HL}}).
In this article we prove the existence of a semistable sheaf of fixed rank and determinant on a regular fibred surface satisfying some conditions.
Recall that a curve is called \emph{semistable} if it is reduced, connected, has only nodal singularities, 
all of whose irreducible components which are rational meet the other components in at least $2$ points.
We call a curve \emph{generalised tree-like}\footnote{In general, the irreducible components of a tree-like curve are smooth.
We use the term `generalised tree-like curve' to emphasise that the irreducible components may be singular.} if it is semistable 
and after ignoring the singularities of the individual components, the dual graph associated to the curve does not have any loops or cycles i.e. it is a tree.

Throughout this article we use the following notations.

\begin{note}\label{n}
Let $R$ be an excellent Henselian discrete valuation ring with algebraically closed residue field $k$ of any characteristic.
Fix integers $r,d$ with $r\ge2$. Let $X_R$ be a regular fibred surface over $\msp(R)$ with special fibre denoted $X_k$, a generalised
tree-like curve of genus $g \ge 2$. Let $\mc{L}_R$ be a line bundle on $X_R$ of degree $d$. 
Assume that the degree of the restriction of $\mc{L}_R$ on the rational components of $X_k$ is a multiple of $r$.
\end{note}   

In this article we prove the following:

\begin{thm}[see Theorem \ref{foundfk}]\label{thmin2}
Keep Notation \ref{n}.
There exists a rank $r$ locally free sheaf on $X_R$ with determinant $\mc{L}_R$ such that it is semistable on the fibres.
\end{thm}

It should be noted that if the fibred surface in Notation \ref{n} is in fact smooth then one can show that the moduli space 
of rank $r$ stable torsion-free sheaves with fixed determinant on $X_R$ is smooth (see \cite{ink3}) and 
Theorem \ref{thmin2} follows easily using the Henselian property. 
However if the special fibre is singular,
then we cannot even define a moduli space of semistable sheaves with fixed determinant.
This is because a locally-free resolution of a coherent sheaf on a singular variety is not necessarily finite. 
So, the usual definition of determinant does not extend to coherent sheaves over singular varieties. 
One could define the determinant to be the top wedge product of the sheaf, but in that case the determinant could have torsion. 
In \cite{B1}, Bhosle defines determinant of a torsion-free sheaf on an irreducible 
nodal curve as the top wedge product modulo torsion. However, she points out in \cite[Remark $4.8$]{B1} that it is not possible
to define a determinant morphism, using this definition, on the entire moduli space of semi-stable 
sheaves due to obvious ambiguity of degree. 
It may be possible to prove a version of Theorem \ref{thmin2} using degeneration of moduli spaces of semistable sheaves of
fixed rank and degree (without fixing determinant) as studied in \cite{nagsesh},\cite{sch}. 
However in all of these articles $R$ is a $k$-algebra. 
When $R$ is of mixed characteristic, these techniques fail (see for example {\cite[Proposition $8$]{nagsesh}}). 
In this article we circumvent these problems by using a direct approach in which we do not need the existence 
of a moduli space of semistable locally free sheaves with fixed determinant and 
instead we use the existence of semistable locally free sheaves with fixed determinant on singular curves.   

\vspace{0.2 cm}
More precisely, our strategy to prove Theorem \ref{thmin2} is to first prove the existence of a semistable locally free sheaf on the special fibre $X_k$
of rank $r$ with determinant $\mc{L}_k := \mc{L}_R \otimes_{R}k$. 
For this we prove the existence of rank $r$ semistable locally free sheaves with fixed determinant on each of the irreducible components of $X_{k}$.
By assumption $X_{k}$ is a semistable curve of genus $g\geq 2$.
Therefore a component of $X_k$, say $Y$ is either a rational curve, an elliptic curve, a smooth or an irreducible nodal curve of genus $g \geq 2$. 
For $Y$ a rational curve, the existence of a semistable locally free sheaf of fixed rank and determinant follows from the assumption in Notation $\ref{n}$. 
In the case $Y$ is a curve of genus $1$, the existence of a locally free sheaf with fixed rank and determinant follows from {\cite[Theorem $3$]{Tu}}.
In {\cite[Theorem $2.4.6$]{IK}}, we replace certain steps (which fail in positive characteristic) in the proof of \cite[Theorem $8.6.1$]{po} to show the existence of a semistable sheaf of rank $r$ and degree $d$ on a smooth curve $Y$ of genus $g\geq 2$ over an algebraically closed field of any characteristic. 
Then using the fact that $\mr{Pic}^d(Y)$ is an abelian variety and that multiplication by rank is an isogeny one concludes that there in fact exists a semistable locally free sheaf of rank $r$ and fixed determinant.
We use the theory of generalised parabolic bundles and the case of the smooth curve in arbitrary characteristic to prove the existence of a semistable locally free sheaf with fixed rank and determinant for an irreducible nodal curve over an algebraically closed field of arbitrary characteristic.
Finally we prove that the glued-up sheaf satisfies the conditions for semistability of locally free sheaves on tree-like curves introduced in  \cite{teix1}.

\begin{thm}[see Theorem \ref{existvbdettl}]\label{thmin1}
Keep Notations \ref{n}.
There exists a Gieseker semistable locally free sheaf $\mc{F}_k$ of rank $r$ on $X_k$ with determinant $\mc{L}_k$.
\end{thm} 

We complete the proof of Theorem \ref{thmin2} using standard techniques such as Grothendieck's formal function theorem (see Proposition \ref{p5})
and Artin approximation to lift the locally free sheaf from the special fibre $X_k$ obtained in Theorem \ref{thmin1} to $X_R$. However the main obstacle is to ensure that the resulting locally free sheaf on $X_R$ has determinant $\mc{L}_R$. For this we develop the obstruction theory to lifting locally free sheaves with fixed determinant (see Theorem \ref{noobsg}) and prove that the obstruction vanishes.

\vspace{0.2 cm}
$\bf{Applications}$: Since locally free sheaves on fibred surfaces are widely studied both in pure mathematics as well as mathematical physics, Theorem  \ref{thmin2} has several applications
especially as we work in arbitrary characteristic. Here we state two applications based on recent works. 

\vspace{0.2 cm}
\noindent Recall that a field $L$ is called $C_1$ if any degree $d$ polynomial in $n$ variables with $n>d$ has a non-trivial solution. 
Let $X_{L}$ be a smooth curve of genus $g \geq 2$ defined over the fraction field of a Henselian discrete valuation ring $R$.
Denote by $\mc{L}_{L}$ a line bundle of degree $d$ on $X_{L}$ and by $M_{L}(r, \mc{L}_{L})$ the moduli space of stable locally free sheaves on $X_{L}$ with rank $r$ and determinant $\mc{L}_L$.
Recall that the variety $M_{L}(r, \mc{L}_{L})$ is a Fano variety and is therefore rationally connected.
The $C_{1}$ conjecture due to Lang, Manin and Koll\'{a}r states that every separably rationally connected variety over 
a $C_{1}$ field has a rational point.
The conjecture is still open in the case when $L$ is the fraction field of an Henselian discrete valuation ring 
of characteristic $0$ with algebraically closed residue field of characteristic $p>0$. 
Suppose $X_L$ has a semistable model (in the sense of Deligne-Mumford) say $X_R$ such that its special fibre is a generalised tree-like curve.
By {\cite[Lemma $3.2.4$]{IK}} there exists a lift $\mc{L}_R$ of $\mc{L}_L$ to $X_{R}$ satisfying the conditions given in Notations \ref{n}.
Since a Henselian discrete valuation ring of characteristic $0$ is excellent, 
Theorem \ref{thmin2} gives an alternate proof of the conjecture for the variety $M_{L}(r, \mc{L}_{L})$ to that given 
in \cite{K2}.

\noindent Let $R$ be a discrete valuation ring with algebraically closed residue field denoted $k$ of arbitrary characteristic. 
Let $f:X \to S$ be a  fibred surface with  $S:= \msp(R)$ and the special fibre $X_{k} := X \times_{S}\msp(k)$ a stable singular curve,
with one node.
In \cite{BBN}, Balaji, Barik and Nagaraj construct a relative moduli space of Gieseker-Hitchin pairs, denoted $\mc{G}^{H}_{X/S}(r,d)$ for $(r,d)$ a pair of coprime integers.  
However, the authors do not provide any argument for why $\mc{G}^{H}_{X/S}(r,d)$ is non-empty.
Using Theorem \ref{thmin2} one can show that this moduli space is indeed non-empty.

\vspace{0.2 cm}
\emph{Acknowledgements}: I thank Prof. U. Bhosle and Prof. M. Teixidor i Bigas for discussing their work with me and answering my questions.
I was funded by the Berlin Mathematical School, Germany for the major part of this work and a minor part was completed while I was a visiting scholar at the Tata Institute of Fundamental Research, Mumbai.

\section{Stability for locally-free sheaves on tree-like curves}

In this section we recall the basic definitions and results related to stability of locally free sheaves on reducible curves.

\begin{defi}\label{dfstab}
 Let $X_{k}$ be a projective curve and $\mc{E}$ a coherent sheaf on $X_{k}$.

\begin{enumerate}
\item For $X_{k}$ integral, the \emph{slope} of $\mc{E}$, denoted $\mu(\mc{E})$ is defined as $\frac{d}{r}$, where $d$ is the degree of $\mc{E}$ and $r$ is its rank.   
The sheaf $\mc{E}$ is called $\textit{slope (semi) stable}$ if for any proper subsheaf $\mc{F} \subset \mc{E}$, $\mu(\mc{F}) (\leq) < \mu(\mc{E})$.

\vspace{0.1 cm}

\item Suppose the curve $X_k$ is reducible, say $X_k = \cup_{i=1}^{N} Y_i$, where $Y_i$ denotes an irreducible component. 
Then in \cite{SFV}, C.S Seshadri generalised slope semistability as follows.
Fix a polarisation $\lambda:=(\lambda_1,...,\lambda_N)$ for $0 \le \lambda_i \le 1$ on $X_k$ such that $\sum_i \lambda_i=1$.
Let $\mc{E}$ be a torsion-free sheaf on $X_k$ with rank $r_{i}$ on each component $Y_{i}$. 
Then $\mu_{\mr{sesh}}(\mc{E}) := \frac{\chi(\mc{E})}{\Sigma \lambda_{i}r_{i}}$.  We  call a 
sheaf $\mc{E}$ \emph{Seshadri-(semi)stable with respect to} $\lambda$, if for every subsheaf $\mc{F}$,
$\mu_{\mr{sesh}}(\mc{F}) \leq \mu_{\mr{sesh}}(\mc{E})$.  

\vspace{0.1 cm}

\item Let $\mc{E}$ be a coherent sheaf with support of dimension $d$. The Hilbert polynomial $P(\mc{E})(t)$ of $\mc{E}$ can be expressed as (see {\cite[Lemma $1.2.1$]{HL}}) 
\[P(\mc{E})(t) := \chi(\mc{E}\otimes \mc{O}_{X_k}(t)) =  \displaystyle\sum\limits_{i=0}^{d}\alpha_{i}(\mc{E})\dfrac{t^{i}}{i!} 
   \mbox{ for } t>>0.\] 
   The reduced Hilbert polynomial is denoted $P_{\red}(\mc{E})(t) := \dfrac{P(\mc{E})(t)}{\alpha_{d}(\mc{E})}$ . 
The sheaf $\mc{E}$ is called \emph{Gieseker (semi)stable} if for any proper subsheaf 
$\mc{F} \subset \mc{E}$, $P_{\red}(\mc{F}(t)) (\leq) < P_{\red}(\mc{E}(t))$ for all $t$ large enough. 
\end{enumerate}
\end{defi}
  
For the remainder of this section we will use the following notation.
  
\begin{note}\label{nb1}
Keep Notations \ref{n}. Let $X_k = \cup_{i=1}^{N} Y_i$ be a generalised tree-like curve 
(as defined in the intoduction) with $N$ irreducible components. 
Denote by $X_k^0$ the set of all points of intersection of two irreducible components of $X_k$.
For $Y$ a subcurve of $X_k$ denote by $Y^{0}$ the set of \emph{internal nodes} of $Y$ i.e., the intersection points of any two curves in $Y$.
Denote by $Y^b$ the nodes on $Y$ which are on the \emph{boundary of} $Y$ i.e., internal nodes of $X_k$ 
which lie on an irreducible component of $Y$ but are not internal nodes of $Y$. 
\end{note}

\begin{rem}
 We now recall results on the stability of locally free sheaves on tree-like curves.
 It should be noted that the results stated in \cite{teix1} are 
 under the assumption that the components of the tree-like curve are smooth and none of them are rational. 
 However, for us the irreducible components may be singular or rational. 
 Nonetheless the results quoted here still hold and the reader is referred to {\cite[Appendix $A.4$]{IK}} for the proofs in our setting.    
\end{rem}

\begin{lem}[{\cite[Lemma $1$]{teix1}}]\label{bh26}
Let $X_k$ be as in Notation \ref{nb1}. 
There is an ordering of the components of $X_k$ such that for every $i \le N-1$, there exists at most one connected component of 
$\overline{X_k \backslash Y_i}$ which contains curves with indices greater than $i$. 
Denote by $B(i)$ this connected component and by $G(i):=\overline{X_k \backslash B(i)}$.
Furthermore, $G(i)$ is connected and $G(i)$ intersects $B(i)$ at exactly $1$ point.
 \end{lem}

\begin{note}\label{nb2} 
Observe that Lemma \ref{bh26} implies that for any $1\leq i \leq N$,
$Y_i$ intersects exactly one irreducible curve, say $Y_{\nu(i)}$ in $B(i)$ and $\nu(i)>i$.
Hence, the curve $Y_i \cup Y_{\nu(i)}$ is connected and for any $j<\nu(i), j \not= i $, $Y_i \cup Y_{\nu(i)}$ is contained in $B(j)$.
\begin{enumerate}
\item Fix from now on an ordering on $X_k$ as mentioned in Lemma \ref{bh26}.
For a given irreducible component $Y_i$ of $X_k$, denote by ${\nu(i)}$ the index such that the corresponding curve $Y_{\nu(i)}$ has the property that $Y_{\nu(i)} \subset B(i)$ and intersects $Y_i$ at exactly one point.
\item Let $\mc{E}$ be a locally free sheaf on $X_k$, $\mc{F}$ a subsheaf of $\mc{E}$, denote by $\mc{F}|_{Y_i}$ the image of the natural morphism $\mc{F} \otimes \mo_{Y_i} \to \mc{E} \otimes \mo_{Y_i}$.
For any point $P \in X_k$, again denote by $\mc{F}|_{P}$ the image of the natural morphism $\mc{F} \otimes \mo_P \to \mc{E} \otimes \mo_P$. 
\end{enumerate}
\end{note}

The following lemma shows twisting a locally free sheaf on the whole curve with a divisor coming from its components 
does not change the Euler characteristic of the sheaf.
It is an important technical tool that we use in the proof of Theorem \ref{existvbdettl}. 

\begin{lem}\label{bh27}
Let $Z$ be a connected subcurve of $X_k$ and $\mc{E}$ be a locally free sheaf on $X_k$.
Denote by $Z_1,...,Z_t$ the irreducible components of $Z$.
Fix integers $a_1,...,a_t$ such that $a_i=0$ if $Z_i$ intersect $\ov{X_k \backslash Z}$.
Denote by \[\mc{L}_0:=\mo_{X_R}\left(\sum\limits_{i=1}^t a_i Z_i\right) \otimes \mo_Z.\] 
Then, $\chi\left(\mc{E} \otimes \mo_Z\right)=\chi(\mc{E} \otimes \mc{L}_0 \otimes \mo_Z)$. 
In particular, \[\sum\limits_{i=1}^t \chi(\mc{E} \otimes \mo_{Z_i})=\sum\limits_{i=1}^t \chi(\mc{E} \otimes \mc{L}_0 \otimes \mo_{Z_i}).\]
\end{lem}

\begin{proof}
First observe that it suffices to prove the statement in the case  
$\mc{L}_0 := \mo_{X_R}(Z_j) \otimes \mo_Z$ where $Z_j$ does not intersect $\ov{X_k \backslash Z}$.
Then the argument can be completed by recursion. 
Denote by $\mc{E}_i:= \mc{E} \otimes \mo_{Z_i}$. 
Let $r: =\mr{rk}(\mc{E})$. 
By the following short exact sequences 

\[0 \to \mc{E} \otimes \mo_Z \to \bigoplus_{i=1}^t \mc{E}_i \to \bigoplus\limits_{P \in Z^0} \mo_P^r  \to 0,\]
\[\mbox{ and } \, \, \,  0 \to \mc{E} \otimes \mc{L}_0 \otimes \mo_Z \to \bigoplus_{i=1}^t \mc{E}_i \otimes \mc{L}_0 \to \bigoplus\limits_{P \in Z^0} \mo_P^r  \to 0.\]

\noindent  we notice,
    
\begin{equation}\label{bh08}
 \chi(\mc{E} \otimes \mo_Z)-\chi(\mc{E} \otimes \mc{L}_0 \otimes \mo_Z)=\sum\limits_{i=1}^t\left(\chi(\mc{E}_i)-\chi(\mc{E}_i \otimes \mc{L}_0)\right).
\end{equation}

\noindent Our goal is to show that the right hand side is $0$.   
Since tensor product by an invertible sheaf does not change the rank of a locally free sheaf, $\mr{rk}(\mc{E}_i)=\mr{rk}(\mc{E}_i \otimes \mc{L}_0)$, 
for all $i$.
Using \cite[Proposition $9.1.21$]{QL}, \[\deg(\mc{E}_j \otimes \mo_{X_R}(Z_j)\otimes \mo_Z)=\deg(\mc{E}_j)+r(Z_j^2)=\deg(\mc{E}_j)-r\sum\limits_{i=1, Y_i \not= Z_j}^N Y_i.Z_j\]
which is equal to $\deg(\mc{E}_j)- r\sum\limits_{i=1, i \not= j} Z_iZ_j$ because $Y_i.Z_j=0$ for $Y_i$ not in $Z$. 
 Also,\[ \deg(\mc{E}_i \otimes \mo_{X_R}(Z_j)\otimes \mo_Z)=\deg(\mc{E}_i)+rZ_i.Z_j \mbox{ for } i \not= j.\]
Hence, $\sum\limits_{i=1}^t \deg(\mc{E}_i \otimes \mo_{X_R}(Z_j)\otimes \mo_Z)=\sum\limits_{i=1}^t \deg(\mc{E}_i)$. 
Therefore, $\sum\limits_{i=1}^t \chi(\mc{E}_i)=\sum\limits_{i=1}^t \chi(\mc{E}_i \otimes \mo_{X_R}(Z_j)\otimes \mo_Z)$, which implies the lemma.
 \end{proof}

For tree-like curves we have the following semistability defined by Teixidor i Bigas (see \cite[Inequality $1$]{teix1}). 

\begin{defi}\label{lsemis}
Let $\mc{E}$ be a locally free sheaf of rank $r$ on $X_k$ such that $\mc{E} \otimes \mo_{Y_i}$ is semistable for each $i=1,...,N$. 
Given a polarisation, $\lambda:=(\lambda_1,\lambda_2,...,\lambda_N)$, 
we say that $\mc{E}$ is $\lambda$-\emph{semistable} if for each $i \le N$, the following inequality is satisfied:
   \[(\sum\limits_{Y_j \in G(i)} \lambda_j)\chi(\mc{E}) +r(|G(i)|-1)\le 
   \sum\limits_{Y_j \in G(i)}\chi(\mc{E} \otimes \mo_{Y_j})\le  (\sum\limits_{Y_j \in G(i)} \lambda_j)\chi(\mc{E}) +r|G(i)| \]
   where $|G(i)|$ denotes the number of irreducible components in $G(i)$.
  \end{defi}

The following theorem shows that on a generalised tree-like curve, $\lambda$-semistability is a sufficient criterion for a 
locally free sheaf which is slope semistable on the irreducible components to be Seshadri semistable with respect to the polarisation $\lambda$.

\begin{thm}[{\cite{teix1}}]\label{bh24}
Let $\mc{E}$ be a locally free sheaf on $X_k$ which is $\lambda$-semistable. 
Then $\mc{E}$ is Seshadri semistable with respect to the polarisation $\lambda$.
\end{thm}

\section{Preliminaries on generalised parabolic bundles}\label{existnc}

In this section we briefly recall the basic definitions and results concerning generalised parabolic 
bundles that we need. We state definitions in our setting although they may hold much more generally.

The following is a well-known result. 

\begin{thm}\label{existvbsmdet}
Let $k$ be an algebraically closed field of arbitrary characteristic and $Y$ be a smooth, projective curve of genus $g \geq 1$
defined over $k$.
Fix integers $r,d$ with $r\geq 2$.
Let $\mc{L}$ be an invertible sheaf of degree $d$ on $Y$.
There exists a semistable locally free sheaf of rank $r$ and determinant $\mc{L}$ on $Y$. 
\end{thm}

\begin{proof}
In the case when $Y$ is a curve of genus $g=1$ see \cite[Theorem $3$]{Tu} for a proof.  
For the case $g\geq 2$, see {\cite[Theorem $2.4.6$ and Proposition $2.4.7$]{IK}} for a proof that works in any characteristic.
\end{proof}

\begin{rem}\label{ratmul}
By Grothendieck's theorem, the only semistable locally free sheaves of rank $r$ on $\mathbf{P}^1$ are of the form $\oplus_{i=1}^{r} \mc{O}_{\mathbf{P}^1}(d)$ for some $d \in \mathbf{Z}$ (see for example \cite[Lemma $3.2.2$]{IK} for a full proof). 
\end{rem}

For the rest of this section, we keep the following notations.

\begin{note}\label{n11}
Let $Y$ be an irreducible nodal curve defined over an algebraically closed field of arbitrary characteristic.
Denote by $\pi: \tilde{Y} \rightarrow Y$ the normalisation map.
Let $\mc{Q}$ be an invertible sheaf on $Y$ of degree $d$.
Denote by $J$ the set of nodes of $Y$ and let $\gamma$ be the number of nodes.
For all $1 \leq i \leq \gamma$, let $p_{i},q_{i}$ be the two points in $\tx$ lying over the double point $x_{i} \in Y$ and let $D_{i} := p_{i} + q_{i}$ be an effective divisor on $\tilde{Y}$.
Denote by $\mc{E}$ a semistable locally free sheaf on $\tx$ of rank $r$ and determinant $\pi^{*}Q$, the existence of which we have proven in Theorem \ref{existvbsmdet}.

\vspace{0.2 cm}
Denote by $\mc{E}|_{D_i} := H^{0}(\mc{E}\otimes\mc{O}_{D_{i}}) \otimes \mo_{D_i}$.
For $x_i \in J$, let 
\[\mc{E}(p_i):= \mc{E}_{p_i} \otimes k(p_i), \ \ \  \mc{E}(q_i):= \mc{E}_{q_i} \otimes k(q_i),\] 
\noindent where $k(p_i)$ and $k(q_i)$ are the residue fields at the points $p_i$ and $q_i$ respectively.
Fix a set of basis elements $\{e_j\}_{j=1}^r$ and $\{f_j\}_{j=1}^{r}$ of $\mc{E}_{p_i}$ and $\mc{E}_{q_i}$, respectively.
By abuse of notation, we  again denote by $e_j$ and $f_j$ their image in $\mc{E}(p_i)$ and $\mc{E}(q_i)$, respectively.
\end{note}

\begin{defi}[{\cite[Definition $3.1$]{B1}}]\label{tfgpb}
Let $\mc{E}$ be as in Notation \ref{n11}. Furthermore assume that $\tx$ has genus $g \geq 1$.
\begin{enumerate} 
\item We define a \emph{generalised parabolic structure} $\sigma$ of $\mc{E}$ over the divisors $D_i$ as follows.
Denote by $F^{i}_1(\mc{E})$ the $k$-vector space generated by $e_{j}\oplus f_{j}$ for all $i = 1, \dots \gamma$. We assign to each singular point $x_i$, $1 \leq i \leq \gamma$:  
\begin{enumerate}
\item a flag of vector subspaces $\Lambda^{i}: F^{i}_{0}(\mc{E}) = \mc{E}|_{D_{i}} \supset F^{i}_1(\mc{E}) \supset F^{i}_2(\mc{E}) = 0$.
\item weights $\ul{\alpha}^{i} = (0,1)$. 
\end{enumerate}
\item We define a \emph{generalised parabolic locally free sheaf} $(\mc{E},\underline{\Lambda},\underline{\alpha})$ where 
$\underline{\Lambda} = (\Lambda^{1},\dots,\Lambda^{\gamma})$ and $\underline{\alpha} = (\ul{\alpha}^{1}, \dots ,\ul{\alpha^{\gamma}})$.
We associate to the generalised parabolic bundle $(\mc{E},\underline{\Lambda},\underline{\alpha})$, the torsion-free sheaf $\phi(\mc{E})$ of rank $r$ and degree $d$ on the nodal curve $Y$ as the kernel of the composition:

\begin{equation}\label{ses1}
  \pi_{*}(\mc{E}) \to \bigoplus _{i}^{\gamma}\pi_{*}(\mc{E}) \otimes k(x_i) \to \bigoplus_{i}^{\gamma} \frac{\pi_{*}(\mc{E}) \otimes k(x_i)}{\pi_{*}(F^{i}_1(\mc{E}))} \to 0 
\end{equation} 
\end{enumerate}
\end{defi}

 \vspace{0.2 cm}
 
\begin{defi}[{\cite[Definition $3.5$]{B1}}]\label{gpbstdefi}
Let $(\mc{E}, \underline{\Lambda},\underline{\alpha})$ be a generalised parabolic locally free sheaf with generalised parabolic structures $(\Lambda^{i}, \underline{\alpha}^{i})$ over the divisors $\{D_i\}_{i=1,\dots \gamma}$.
\begin{enumerate}
\item Denote by $m_1^i:= \mr{dim}(F_{0}^i(\mc{E})/F_{1}^i(\mc{E}))$ and $m_2^i:= \mr{dim}(F_{1}^i(\mc{E})/F_{2}^i(\mc{E}))$.
We define the \emph{weight} of $\mc{E}$ \emph{over a divisor $D_i$} as $wt_{D_i}(\mc{E}) = \sum\limits_{l=1}^{2}m_{l}^i\alpha_{l}^i$.
The \emph{weight} of $\mc{E}$, denoted $\mr{wt}(\mc{E})$, is defined as $\sum \limits_{i}wt_{D_{i}}(\mc{E})$. 
\item The \emph{parabolic degree} of $\mc{E}$ is defined as $\mr{par deg}(\mc{E}) = \mr{deg} (\mc{E}) + \mr{wt}(\mc{E})$.
\item The \emph{parabolic slope} of $\mc{E}$, denoted $\mr{par} \mu(\mc{E})$, is defined as $\frac{\mr{par deg}(\mc{E})}{\mr{rank}(\mc{E})}$. 
\item A generalised parabolic bundle $(\mc{E}, \ul{\Lambda},\ul{\alpha})$ is \emph{parabolic semistable} (respectively \emph{parabolic stable}) if for every proper subbundle $(\mc{K})$ of $\mc{E}$, one has $\mr{par} \mu(\mc{K}) \leq \mr{par} \mu(\mc{E})$ (respectively $< \mr{par} \mu(\mc{E})$). 
\end{enumerate}
\end{defi}
  
\vspace{0.2 cm}
  
\begin{lem}\label{phirkdeg}
The torsion-free sheaf $\phi(\mc{E})$ on the curve $Y$ is locally free of rank $r$ and degree $d$. 
\end{lem}

\begin{proof}
Let $(\mc{E}, \ul{\Lambda}, \ul{\alpha})$ be as in Definition \ref{tfgpb} and let
 \[\pr^i_1:F_1^i(\mc{E}) \to \mc{E}(p_i) \oplus \mc{E}(q_i) \to \mc{E}(p_i), \, \, \pr_2^i:F_1^i(\mc{E}) \to \mc{E}(p_i) \oplus \mc{E}(q_i) \to \mc{E}(q_i).\]
\noindent By definition 
\[\pr_1^i:F^{i}_{1}(\mc{E}) \xrightarrow{\simeq}  k(p_i)^{\oplus r} \, \, \mbox{ and } \, \, \pr_2^i:F^{i}_{1}(\mc{E}) \xrightarrow{\simeq}  k(q_i)^{\oplus r}.\] 
be the natural projection maps.
Therefore by {\cite[Proposition $4.3$]{B1}}, the torsion-free sheaf $\phi(\mc{E})$ is locally free. 

\vspace{0.2 cm}
We now show that $\phi(\mc{E})$ has the same rank and degree as $\mc{E}$. 
Since $Y$ is irreducible and $\pi$ is birational, $\mr{rk}(\pi_{*}(\mc{E})) = \mr{rk}(\mc{E})$.
Furthermore, the rank of $\frac{\pi_{*}(\mc{E}) \otimes k(x_i)}{F^{i}_1(\mc{E})} = 0$ since it is supported on points.
Then using the additivity of rank, we have
\[ \mr{rk}(\phi(\mc{E})) = \mr{rk}(\pi_{*}(\mc{E})) = \mr{rk}(\mc{E}). \]
Moreover  $\chi(\pi_{*}(\mc{E})) = \chi(\phi(\mc{E})) + \sum\limits_{i =1}^{\gamma}\mr{dim}(F^{i}_{1}(\mc{E}))$.
Since pushforwards preserve Euler characteristic, using Riemann-Roch for locally free sheaves, we have
\[ r(1-(\rho_{a}(Y)-\gamma)) + \mr{deg}(\mc{E}) = r(1-\rho_{a}(Y)) + \mr{deg}(\phi(\mc{E})) + r\gamma.  \]
\noindent where $\rho_a$ denotes the genus.
Therefore $\mr{deg}(\phi(\mc{E})) = \mr{deg}(\mc{E})$.
As $\mc{E}$ has degree $d$, so does $\phi(\mc{E})$.
\end{proof}

Using this we can prove the following.

\begin{lem}\label{in01}
 Let $\mc{E}$ and $\phi(\mc{E})$ be as above. Then $\pi^*\phi(\mc{E}) \cong \mc{E}$.
\end{lem}

\begin{proof} 
Consider the pull-back of the morphism $\phi(\mc{E}) \to \pi_{*} \mc{E}$ under the normalisation map $\pi$ and the natural map $\pi^{*}\pi_{*}\mc{E} \to \mc{E}$. Denote by $\tau$ the composition 
 \[ \pi^{*}\phi(\mc{E}) \to \pi^{*}\pi_{*}(\mc{E}) \to \mc{E}.  \] 
 
\noindent Denote by $\mc{K}$ the kernel of $\tau$. 
Note that the localization of $\tau$ at $x$ is an isomorphism for all $x \not\in \pi^{-1}(J)$. 
Hence $\mc{K}$ is supported at finitely many points and is therefore a torsion sheaf. 
However, $\pi^{*}\phi(\mc{E})$ is locally free since by Lemma \ref{phirkdeg}, $\phi(\mc{E})$ is locally free and the pull back of a locally free sheaf is locally free.  
Hence $\pi^*\phi(\mc{E})$ cannot contain a non-zero torsion sheaf implying $\mc{K}=0$.
Therefore $\tau$ is injective.
 
\vspace{0.2 cm}
\noindent Since $\tau$ is an isomorphism for all $x \not\in \pi^{-1}(J)$, $\mr{coker}(\tau)$ is a skyscraper sheaf with support in $\pi^{-1}(J)$.
But degree of a non-trivial skyscraper sheaf is strictly positive.
Since degree is additive, by the short exact sequence 
\[ 0 \to \pi^{*}\phi(\mc{E}) \xrightarrow{\tau} \mc{E} \to \mr{coker}(\tau)\to 0 \] 
\noindent $\mr{deg}(\pi^{*}(\phi(\mc{E})) \leq \mr{deg}(\mc{E})$ with strict inequality if $\tau$ is not surjective. Since $\pi$ is the normalisation map, $\deg(\pi^*\phi(\mc{E})) = \deg(\phi(\mc{E}))$. 
By Lemma \ref{phirkdeg}, $\mr{deg}(\phi(\mc{E})) = \mr{deg}(\mc{E})$.
Hence, $\tau$ must be surjective.
Then $\tau$ is an isomorphism and hence $\pi^*\phi(\mc{E}) \cong \mc{E}$ as required.
\end{proof}

\begin{lem}\label{gpbsemistable} 
The generalised parabolic bundle $(\mc{E},\underline{\Lambda},\underline{\alpha})$ defined in Definition \ref{tfgpb} is parabolic semistable.  
\end{lem}

\begin{proof}
By definition the generalised parabolic bundle $(\mc{E},\underline{\Lambda},\underline{\alpha})$ is parabolic semistable if for any sub-bundle $\mc{K} \subset \mc{E}$ with the induced parabolic structure
$\mr{par} \mu(\mc{K}) \leq \mr{par}\mu(\mc{E})$. 

\noindent Since we take weights $\ul{\alpha}^{i} = (0, 1)$ for all $1 \leq i \leq \gamma$, we have 

\[ \mr{par}\mu(\mc{E}) = \frac{\mr{deg}(\mc{E}) + \gamma(\mr{rk}(\mc{E}))}{\mr{rk}(\mc{E})}. \]

\noindent Note that $\mr{dim}\left(\frac{F^{i}_1(\mc{K})}{F_2^{i}(\mc{K})}\right) = \mr{dim}(F^{i}_{1}(\mc{K}))$ because $F^{i}_{2}(\mc{E}) = 0$, 
where \[F_j^i(\mc{K}_1)=F_j^i(\mc{E}) \cap (\pi_*\mc{K}_1 \otimes k(x_i)) \, \, \mbox{ for } \, j=1,2.\]
Moreover 
\[\mr{dim}(F^{i}_{1}(\mc{K})) = \mr{dim}(F^{i}_1(\mc{E})\cap H^{0}(\mc{K}\otimes \mc{O}_{D_{i}})) \leq \mr{rk}(\mc{K}).\] 
\noindent Therefore 
\[ \mr{par}\mu(\mc{K}) \leq \frac{\mr{deg}(\mc{K}) + \gamma(\mr{rk}(\mc{K}))}{\mr{rk}(\mc{K})}. \]
Since $\mc{E}$ is a semistable locally free sheaf we have 
\[ \mr{par} \mu(\mc{K}) \leq \frac{\mr{deg}(\mc{K}) + \gamma(\mr{rk}(\mc{K}))}{\mr{rk}(\mc{K})} \leq \frac{\mr{deg}(\mc{E}) + \gamma(\mr{rk}(\mc{E}))}{\mr{rk}(\mc{E})} = \mr{par} \mu(\mc{E})\]
\noindent and therefore $(\mc{E},\underline{\Lambda},\underline{\alpha})$ is a semistable generalised parabolic bundle. 
\end{proof}

\section{Existence of semistable locally free sheaves with fixed determinants on singular curves}\label{liftgroth}

Keep notations from previous sections.
In this section we prove the existence of a Gieseker semistable locally free sheaf $\mc{F}_k$ 
on $X_k$ such that $\mr{det}(\mc{F}_k) \simeq \mc{L}_{k}$ (see Theorem \ref{existvbdettl}). 
We first produce a semistable locally free sheaf on an irreducible nodal curve.

\vspace{0.2 cm}
In the following proposition we consider the case when the normalisation of an irreducible nodal curve is a smooth rational curve.

\begin{prop}\label{theratcase}
Let $Y$ be an irreducible nodal curve defined over an algebraically closed field such that the genus of $\tx$ is $0$. 
Then there exists a semistable locally free sheaf of rank $r$ and degree $d$ on $Y$.
\end{prop}

\begin{proof}
For simplicity we assume that the irreducible nodal curve has only one node. The general case follows similarly.
Let $a \in \mathbb{Z}$ such that $ra \le d$.
Consider the locally free sheaf $\mathcal{E} := \mathcal{O}_{\tilde{Y}}(d-(r-1)a) \oplus \mc{L}_1 \oplus ... \oplus \mc{L}_{r-1}$, where $\mc{L}_j=\mathcal{O}_{\tilde{Y}}(a)$. 
Let $x$ be the nodal point on $Y$ and $p,q$ the points on $\tilde{Y}$ over $x$ .
Denote by $D:=p+q$, $e_1$ (resp. $f_1$) a basis of $\mathcal{O}_{\tilde{Y}}(d-(r-1)a)(p)$ 
(resp. $\mathcal{O}_{\tilde{Y}}(d-(r-1)a)(q)$) the fibers at $p$ and $q$ respectively of the line bundle $\mathcal{O}_{\tilde{Y}}(d-(r-1)a)$. 
Similarly, fix a basis $e_{j+1}$ (resp. $f_{j+1}$) of the fiber $\mc{L}_j(p)$ (resp. $\mc{L}_j(q)$) for $1\leq j \leq r-1$. Choose $F_1(\mathcal{E})$ to be the vector subspace of the skyscaper sheaf $\mathcal{E}(p) \oplus \mathcal{E}(q)$ generated by 
$e_j \otimes \mathcal{O}_p \oplus (f_1 \oplus ... \oplus f_{j-1} \oplus \hat{f_j} \oplus f_{j+1} ... \oplus f_r) \otimes \mathcal{O}_q$ for $1\leq j \leq r$.
Let
 \[\pr_1:F_1(\mc{E}) \to \mc{E}(p) \oplus \mc{E}(q) \to \mc{E}(p), \, \, \pr_2:F_1(\mc{E}) \to \mc{E}(p) \oplus \mc{E}(q) \to \mc{E}(q)\]
 \noindent be the natural projection maps.
It is easy to check $\pr_1$ and $\pr_2$ are isomorphisms.  
Then by arguments exactly as in the proof of Lemma \ref{phirkdeg}, $\phi(\mc{E})$ is locally free of rank $r$ and degree $d$.

\vspace{0.2 cm}
Let us check that $\phi(\mathcal{E})$ is semistable. 
For this we need to prove that for any subsheaf $\mathcal{F} \subset \mathcal{E}$ with $\mathcal{F}(D) \subset F_1(E)$ (by $\mc{F}(D)$
we mean the fiber of $\mc{F}$ over $D$), we have 
\[\deg(\phi(\mc{F}))/\mathrm{rk}(\phi(\mc(F))) \leq \deg(\phi(\mc{E}))/\mathrm{rk}(\phi(\mc{E})) = d/r.\]
Again by arguments as in Lemma \ref{phirkdeg}, $\deg(\phi(\mc{F})) = \deg(\mc{F})$ and $\mr{rk}(\phi(\mc{F})) = \mr{rk}(\mc{F})$. 
Hence it suffices to prove $\deg(\mc{F})/\mr{rk}(\mc{F}) \leq d/r$. We claim that for any $\mc{F} \subset \mc{E}$ 
with $\mc{F}(D) \subset F_1(\mc{E})$, the composition map 
\[\psi:\mathcal{F} \hookrightarrow \mathcal{E} \to \mc{L}_1 \oplus ... \oplus \mc{L}_{r-1}\] is injective, 
\noindent where the last map is simply the projection map. 
Since $\mc{L}_1 \oplus ... \oplus \mc{L}_{r-1}$ is semistable, this will imply
\[\deg(\mc{F})/\mr{rk}(\mc{F}) \leq \mr{deg}(\mc{L}_1) = a.\] Since, by assumption $a \leq d/r$, 
 we will have the required result.

\emph{Proof of Claim}
Suppose $\psi$ is not injective. Then $\ker \psi$ is a non-zero locally free sheaf since $\tx$ is smooth
and $\mc{F}$ is locally free. This means there exists a small enough open neighbourhood $U$ of $D$ in $\tilde{Y}$ and at least one section $s \in \ker \psi|_U$
such that $s(p) \not=0$. As $\psi(s)=0$, this means $s(p)=\lambda e_1$, $\lambda \not=0$, which implies $s(q)=\lambda(f_2 \oplus ... \oplus f_r)$, as $s(D) \in F_1(\mathcal{E})$.
But by assumption, $\psi(s)(q)=0$ which is a contradiction. This means $\psi$ is injective. 
\end{proof}

For the case when $Y$ is an irreducible nodal curve of genus $g\geq 2$ with a single node, defined over the complex numbers, 
X. Sun proves the following theorem in \cite[Theorem $1$]{sun1}.

\begin{thm}\label{existvb} 
Let $Y$ be an irreducible nodal curve defined over an algebraically closed field $k$ of arbitrary characteristic with normalisation $\tx$ of genus $g\geq0$.
Let $\mc{Q}$ be an invertible sheaf on $Y$ of degree $d$.
There exists a semistable locally free sheaf on $Y$ of rank $r$ and determinant $\mc{Q}$.
\end{thm}

\begin{proof}
If $\tx$ has genus $0$, then there exists a semistable locally free sheaf $\mc{F}$ on $Y$ by Proposition \ref{theratcase}. 
By Grothendieck's theorem, in this case, $\pi^*\mc{Q} \cong \mo_{\tilde{Y}}(d) \cong \pi^*\det \mc{F}$.
For $\tx$ having genus $g \geq 1$, by Theorem \ref{existvbsmdet}, there exists a semistable locally free sheaf $\mc{E}$ of rank $r$ and determinant $\pi^*\mc{Q}$ on $\tx$.
Let $(\mc{E},\underline{\Lambda},\underline{\alpha})$ be the generalised parabolic bundle as in Definition \ref{tfgpb} and $\phi(\mc{E})$ be the corresponding torsion-free sheaf.
By Lemma \ref{phirkdeg}, this is a locally free sheaf on $Y$ of rank $r$ and degree $d$.
The generalised parabolic bundle $(\mc{E},\underline{\Lambda},\underline{\alpha})$ is semistable by Lemma \ref{gpbsemistable}.
Then by {\cite[Proposition $4.2$]{B1}}, $\phi(\mc{E})$ is semistable. Denote by $\mc{F}:=\phi(\mc{E})$. Observe by Lemma \ref{in01}
that $\pi^*\det \mc{F} \cong \pi^*\mc{Q}$.

\vspace{0.2 cm}
For any $x \in Y$, denote by $\tilde{\mathcal{O}}_{Y,x}$ the integral closure of $\mathcal{O}_{Y,x}$. 
For any $x_i \in J$, $\tilde{\mathcal{O}}_{Y,x_i}^*/\mc{O}_{Y,x_i}^* \cong k^*$.
By \cite[Ex. II.$6.9$]{H}, we have a short exact sequence 
  
\begin{equation}\label{in10}
0 \to \oplus_{x_i \in J} k^* \to \mr{Pic}(Y) \xrightarrow{\pi^*} \mr{Pic}(\tx) \to 0.
\end{equation}

\noindent Since pull-back commutes with tensor product, \[\pi^*(\mc{Q} \otimes \det(\mc{F})^{-1})  \cong \pi^{*}(\mc{Q}) \otimes \pi^{*}(\mr{det}(\mc{F})^{-1}) \cong \mo_{\tx}.\]
Therefore, the invertible sheaf $\mc{R} := \mc{Q} \otimes_{\mc{O}_{Y}}(\mr{det}(\mc{F})^{-1})$ is in the kernel of $\pi^{*}$.
As $k$ is algebraically closed, the morphism 
\[[r]:\oplus_{x_i \in Q} k^* \to \oplus_{x_i \in Q} k^* ; \ \ \ [r](a_1,a_2, \dots a_n) = (a_1^r, a_2^r, \dots a^r_n)\] 
\noindent is surjective.
Therefore there exists an invertible sheaf, say $\mc{R}'$ on $Y$ such that $\mc{R} \cong \mc{R}'^{\otimes r}$. 

\vspace{0.2 cm}
Let $\mc{G} := \mc{F} \otimes \mc{R}'$.
It is easy to see that $\mc{G}$ is locally free and as twisting with an invertible sheaf does not change local freeness or the rank, $\mr{rank}(\mc{G}) = r$.
Since stability is also preserved under twisting with an invertible sheaf, $\mc{G}$ is semistable. 
Moreover, 
\[\det(\mc{G}) \cong \det(\mc{F} \otimes \mc{R}') \cong \det(\mc{F}) \otimes \mc{R}'^{\otimes r} \cong \det(\mc{F}) \otimes \mc{R} \cong \mc{Q}.\]
The degree of a locally free sheaf is the same as that of its determinant, hence $\mr{deg}(\mc{G}) = d$. 
This proves the theorem. 
\end{proof}

\begin{note}\label{n3}
Keep Notations \ref{n} and \ref{nb1}, Lemma \ref{bh26} and Notations \ref{nb2}. 
Denote by $Y_1,\dots,Y_N$ the irreducible components of the generalised tree-like curve $X_k$. 
For $1 \leq i \leq N$, denote by $Y_{\nu(i)}$ the unique component in $B(i)$ which intersects $Y_i$.
\end{note}

\vspace{0.2 cm}
Now we prove the main result of this section.

\vspace{0.2 cm}

\begin{thm}\label{existvbdettl}
Keep Notations \ref{n3}.
There exists a Gieseker semistable locally free sheaf $\mc{F}_k$ of 
rank $r$ on $X_k$ with $\det(\mc{F}_k) \cong \mc{L}_k$.
\end{thm}

\begin{proof}
By assumption $X_k = \cup_{i=1}^{N} Y_i$ is a  generalised tree-like curve. 
Hence the irreducible components $Y_i$ are rational, smooth or irreducible nodal curves. 
Denote by $\mc{L}_i := \mc{L}_k |_{Y_i}$. 
In the case $Y_i$ is rational, there exists a slope semistable locally free sheaf $\mc{F}_i$  of rank $r$ 
with determinant $\mc{L}_i$ on $Y_i$ since by assumption the degree of the restriction of $\mc{L}_k$ on the rational components is a multiple of $r$ (see Remark \ref{ratmul}).
If a component $Y_i$ is smooth (resp. irreducible nodal) there exists a slope semistable locally free sheaf $\mc{E}_i$  of rank $r$ 
with determinant $\mc{L}_i$ on $Y_i$  by Theorem \ref{existvbsmdet} (resp. Theorem  \ref{existvb}). 
Define $\mc{F}_k$ to be the locally free sheaf on $X_k$ obtained by glueing $\mc{F}_i$ for $1\leq i \leq N$, on the intersection points.
Then $\det(\mc{F}_k) = \mc{L}_k$. 
Our goal is to prove that $\mc{F}_k$ is in fact Gieseker semistable.

\vspace{0.2 cm}
We first show that for \emph{any} polarisation $\lambda:=(\lambda_1,\lambda_2,...,\lambda_N)$, there exist integers $a_1, a_2, \dots a_N$ such that  
$\mc{F}'_k :=\mc{F}_{k} \otimes_{R} \mo_{X_R}(\sum\limits_{i=1}^N a_i Y_i)$ is Seshadri semistable with respect to the polarisation $\lambda$.

\vspace{0.2 cm}
We now prove by recursion the existence of integers $a_1, a_2, \dots a_N$ such that  
$\mc{F}_{k} \otimes_{R} \mo_{X_R}(\sum\limits_{i=1}^N a_i Y_i)$ is $\lambda$-semistable i.e for all $1 \leq i \leq N$, $\mc{F}_{k} \otimes \mo_{X_R}(\sum\limits_{i=1}^N a_i Y_i)$,
satisfies the following inequality: 

\[ (\sum\limits_{Y_j \in G(i)} \lambda_j)\chi(\mc{F}_{k} \otimes \mo_{X_R}(\sum\limits_{i=1}^N a_i Y_i)) + r(|G(i)|-1) \le 
   \sum\limits_{Y_j \in G(i)} \chi(\mc{F}_{k} \otimes \mo_{X_R}(\sum\limits_{i=1}^N a_i Y_i) \otimes \mo_{Y_j}) \le \]
 \[\le (\sum\limits_{Y_j \in G(i)} \lambda_j)\chi(\mc{F}_{k} \otimes \mo_{X_R}(\sum\limits_{i=1}^N a_i Y_i)) +r|G(i)|.     \ \ \ \ (*) \]
 
 \vspace{0.2 cm}
\noindent By Theorem \ref{bh24} if a locally free sheaf is $\lambda$-semistable, then it is Seshadri semistable, so it suffices to obtain a $\lambda$-semistable sheaf.

\vspace{0.2 cm}
Note that as a consequence of the ordering given in Lemma \ref{bh26}, any locally free sheaf $\mc{F}_{k}$ is $\lambda$-semistable on $Y_N$ for any $a_N$.
Indeed, for  $i = N$, we have $\sum\limits_{i=1}^N \lambda_i=1$, $G(N)=X_k$ and $|G(N)|= N$. 
Hence the inequality $(*)$ can be written as,

 \[ (\sum\limits_{Y_j \in X_k} \lambda_j)\chi(\mc{F}_{k} \otimes \mo_{X_R}(a_N Y_N)) + r(N-1) \le 
   \sum\limits_{Y_j \in X_{k}} \chi(\mc{F}_{k} \otimes \mo_{X_R}(a_N Y_N) \otimes \mo_{Y_j}) \le \]
 \[\le (\sum\limits_{Y_j \in X_k} \lambda_j)\chi(\mc{F}_{k}\otimes \mo_{X_R}(a_N Y_N)) +rN.\] 
 
\noindent Moreover
\[ \sum\limits_{i=1}^N\chi( \mc{F}_{k} \otimes \mo_{X_R}(a_N Y_N) \otimes \mo_{Y_i}) = \chi(\mc{F}_{k}\otimes \mo_{X_R}(a_N Y_N)) + \sum\limits_{i=1}^{N-1} \chi(\mc{F}_{k}\otimes \mo_{X_R}(a_N Y_N)\otimes \mc{O}_{P_i}).\]
\noindent where $P_{i}$ are the nodes in $X_k$ connecting two irreducible components. 
Since there are $N-1$ such nodes and $\mr{dim}(\mc{F}_{k}\otimes \mo_{P_i}) = r$ for all $i$, 
we have by Lemma \ref{bh27} that
\[ \sum\limits_{i=1}^N\chi( \mc{F}_{k} \otimes \mo_{X_R}(a_N Y_N) \otimes \mo_{Y_i}) = \chi(\mc{F}_{k})+r(N-1).\]
Since $\chi(\mc{F}_{k} \otimes \mo_{X_R}(a_N Y_N)) = \chi(\mc{F}_{k})$ (Lemma \ref{bh27}), 
the inequality above is satisfied
for any $a_N$.
Therefore $\mc{F}_{k}\otimes \mo_{X_R}(a_N Y_N)$ is $\lambda$-semistable on $Y_N$.

\vspace{0.2 cm}
Assume that for some $n_0 \leq N$, $\mc{F}_{k}$ is $\lambda$-semistable on $Y_{n_0+1}, \dots Y_{N}$.
Therefore for $i = n_0$, we have 

\[\sum\limits_{Y_j \in G(n_0)} \deg(\mc{F}_{k} \otimes \mo_{X_R}(a_{n_0}Y_{n_0}) \otimes \mo_{Y_j})= \sum\limits_{Y_j \in G(n_0)} \left(\deg(\mc{F}_{k} \otimes \mo_{Y_{j}})+ ra_{n_0}Y_{n_0}.Y_j\right).\]

\noindent Note that for any $i$, $Y_{i}.Y_{\nu(i)} = 1$ and $\nu(i) > i$. By \cite[Proposition $9.1.21$]{QL}, $Y_{i}^2 = -Y_{i}.Y_{\nu(i)}-\sum\limits_{Y_j \in G(i)\backslash Y_{i}} Y_j.Y_{i}$. 
Therefore for $i = n_0$ we have   
\[\sum\limits_{Y_j \in G(n_0)} \left(\deg(\mc{F}_{k} \otimes \mo_{Y_j})+ ra_{n_0}Y_{n_0}.Y_j\right)
= \left(\sum\limits_{Y_j \in G(n_0)} \deg(\mc{F}_{k} \otimes \mo_{Y_j})\right)-ra_{n_0}.\]

\noindent Since the rank of $\mc{F}_{k}$ does not change after twisting with a invertible sheaf, the Euler characteristic depends only on the degree.
Furthermore for any $i$, the difference between the upper-bound and the lower bound in the inequality $(*)$ is equal to $r$.
Then by the Euclidean algorithm and Lemma \ref{bh27} there must exist an integer $a_{n_0}$ such that 
\[ (\sum\limits_{Y_j \in G(n_0)} \lambda_j)\chi(\mc{F}_{k} \otimes \mo_{X_R}(a_{n_0} Y_{n_0})) + r(|G(n_0)|-1) \le \]
 \[ \le  \sum\limits_{Y_j \in G(n_0)} \chi(\mc{F}_{k} \otimes \mo_{X_R}( a_{n_0} Y_{n_0}) \otimes \mo_{Y_j}) 
 \le (\sum\limits_{Y_j \in G(n_0)} \lambda_j)\chi(\mc{F}_{k} \otimes \mo_{X_R}(a_{n_0} Y_{n_0})) +r|G(n_0)|.  \]

\noindent Hence, $\mc{F}_{k} \otimes \mo_{X_R}(a_{n_0}Y_{n_0})$ is $\lambda$-semistable at $Y_{n_0}$.

\vspace{0.2 cm}
Note that $\mc{F}_{k} \otimes \mo_{X_R}(a_{n_0}Y_{n_0})$ is also $\lambda$-semistable on $Y_t$ for all $n_0<t<N$.
This is because $Y_{n_0}$ does not intersect any curve in $\ov{B(n_0)\backslash Y_{\nu(n_0)}}$, neither does any curve in $G(n_0)$ (the only curve in $G(n_0)$ that intersects $B(n_0)$ is $Y_{n_0}$).
By Lemma \ref{bh27}, this implies 

\[\sum\limits_{Y_j \in G(n_0) \cup Y_{\nu(n_0)}} \chi(\mc{F}_{k} \otimes \mo_{X_R}(a_{n_0}Y_{n_0}) \otimes \mo_{Y_j})=
\sum\limits_{Y_j \in G(n_0) \cup Y_{\nu(n_0)}} \chi(\mc{F}_{k} \otimes \mo_{Y_j}) .\]

\vspace{0.2 cm}
Furthermore, $\chi(\mc{F}_{k} \otimes \mo_{X_R}(a_{n_0}Y_{n_0}) \otimes \mo_{Y_j})= \chi(\mc{F}_{k} \otimes \mo_{Y_j})$ for any curve $Y_j \in \ov{B(n_0)\backslash Y_{\nu(n_0)}}$. 
Note that $G(n_0) \cup Y_{\nu(n_0)}$ is connected.
Hence, for any $t>n_0$, either $G(n_0) \cup Y_{\nu(n_0)}$ is entirely contained in $G(t)$ or in $B(t)$.
Therefore for  any $t>n_0$,
\[\sum\limits_{Y_j \in G(t)} \chi(\mc{F}_{k} \otimes \mo_{X_R}(a_{n_0}Y_{n_0})\otimes \mo_{Y_j})=\sum\limits_{Y_j \in G(t)} \chi(\mc{F}_{k} \otimes \mo_{Y_j}).\]
By hypothesis, $\mc{F}_{k}$ is $\lambda$-semistable at $Y_t$ for all $t>n_0$. 
Furthermore by Lemma \ref{bh27}, we have $\chi(\mc{F}_{k}) = \chi(\mc{F}_{k} \otimes \mo_{X_R}(a_{n_0}Y_{n_0}))$.
Since the sum of the Euler characteristics of $\mc{F}_{k}$ when restricted to curves in $G(t)$ is the same as that of $\mc{F}_{k} \otimes \mo_{X_R}(a_{n_0}Y_{n_0})$, this implies $\mc{F}_{k} \otimes \mo_{X_R}(a_{n_0}Y_{n_0})$ is also $\lambda$-semistable for all $t>n_0$. 

\vspace{0.2 cm} 
By recursion we can find integers $a_{i}$ for all $1\leq i<N$ such that $\mc{F}_{k} \otimes \mo_{X_R}(\sum\limits_{i=1}^{N} a_{i}Y_{i})$ is $\lambda$-semistable on $Y_i$. 
Therefore the locally free sheaf $\mc{F}_{k}':= \mc{F}_{k} \otimes \mo_{X_R}(\sum\limits_{i=1}^{N} a_{i}Y_{i})$ is $\lambda$-semistable. 
Finally by Theorem \ref{bh24} it is also Seshadri semistable with respect to the polarisation $\lambda$.
Now set $\lambda_{i} = \frac{\alpha_{1}(\mc{O}_{Y_i})} {\alpha_{1}(\mc{O}_{X_k})}$
where $\alpha_1(-)$ is the leading coefficient of the corresponding Hilbert polynomial $P(-)(t) :=  \displaystyle\sum\limits_{i=0}^{d}\alpha_{i}(-)\dfrac{t^{i}}{i!}$.
It is a standard computation to check that with this choice of $\lambda_{i}$, $\mc{F}'_{k}$ is also Gieseker semistable (see for example \cite[Lemma $3.2.9$]{IK}). 
However the determinant of $\mc{F}_k^{'}$ is not $\mc{L}_k$. 
Since twisting with an invertible sheaf does not change Gieseker semistability, we can twist $\mc{F}_{k}'$ with the dual of $\mo_{X_R}(\sum\limits_{i=1}^{N} a_{i}Y_{i})$.
This proves that the sheaf $\mc{F}_k$ was Gieseker semistable to begin with.
\end{proof}

\section{Existence of locally free sheaves with fixed determinant on fibred surface}

\begin{note}\label{n2}
Keep Notations \ref{n}. 
In this section semistable always means Gieseker semistable.
Denote by $\hat{R}$ the completion of the discrete valuation ring $R$, by $\pi$ the uniformizer of $R$ and by $X_{\hat{R}}:= X_R \otimes_{R} \hat{R}$.  
Let $\mc{L}_{\hat{R}}:= \mc{L}_R \otimes _{R} \hat{R}$ and $\mc{L}_k := \mc{L}_{R}\otimes_{R}k$. 
Denote by $\mc{F}_k$ the semistable locally free sheaf on $X_{k}$ of rank $r$ and determinant $\mc{L}_k$ obtained in Theorem \ref{existvbdettl}.
For $n \geq 1$, let $R_{n} := R/\mf{m}^{n+1}$ where $\mf{m}$ denotes the maximal ideal of the ring $R$.
Denote by  $Y_{n}$, the spectrum of the ring $R_{n}$.
Let $X_{n} := X_{\hat{R}}\times_{\hat{R}}\msp R_{n}$. 
Since $X_{\hat{R}}$ is flat over $\msp (\hat{R})$, the scheme $X_{n}$ is flat over $\msp R_{n}$ and $X_{k}\simeq X_{n}\times_{R_{n}} \msp k$.
Therefore $X_{n}$ is a deformation of $X_{k}$ over $R_{n}$.
\end{note}

In this section we prove the existence of a semistable locally free sheaf 
$\mc{F}_R$ with determinant $\mc{L}_R$ on $X_{R}$ (see Theorem \ref{foundfk}). 
Using Grothendieck's formal function theorem, we first prove the existence of a locally free sheaf with 
determinant $\mc{L}_{\hat{R}}$ on $X_{\hat{R}}$ (see Proposition \ref{p5}). Note that for this we require the underlying ring to be complete (see \cite[Proposition $9.6$]{H} and \cite[Theorem $8.4.2$]{FGA}), therefore we use $\hat{R}$ instead of $R$. To ensure $\det(\mc{F}_R)$ is $\mc{L}_R$ we develop the obstruction theory to lifting the locally free sheaf $\mc{F}_k$ on $X_k$ with determinant $\mc{L}_k$ to a locally free sheaf on $X_n$ with determinant $\mc{L}_{\hat{R}|_{X_n}}$ for any $n>0$ (see Theorem \ref{noobsg}). Finally using Artin approximation and properties of semistability of sheaves, we obtain a locally free sheaf $\mc{F}_R$ on $X_R$ with determinant $\mc{L}_R$ which is semistable on geometric fibers.

\vspace{0.2 cm}
Recall the following general definitions. 

\begin{defi}\label{fsr}
We define a ringed space called the $\textit{formal spectrum}$ of $R$, 
denoted $\mc{Y} := (\mr{Spf}(R),\mc{O}_{\mc{Y}})$ as follows:
The topological space $\mr{Spf}(R)$ consists of the closed point $\mathrm{Spec}(k)$, with the discrete topology and the sheaf of rings $\mc{O}_{\mc{Y}}$ is $\hat{R}$.
By \cite[Theorem $21.1$]{HD} there is a noetherian formal scheme $\mc{X}$, flat over $\mc{Y}$, such that for each $n$, $X_{n} \simeq \mc{X}\times_{\hat{R}} Y_{n}$. 
\end{defi}

\vspace{0.2 cm}
We now discuss how the extensions of a locally free sheaf on $X_n$ relate to those of its determinant bundle.

\begin{defi} We have the following definitions:
\begin{enumerate}
\item The \emph{trace map} $\mr{tr}: M(r,\mc{O}_{X_{n+1}}) \to \mc{O}_{X_{n+1}}$ is given by taking the trace of the matrices.
\item We define the map 
 \begin{align*}
\pi^{n}\mr{tr}:\pi^{n}M(r, \mc{O}_{X_{n+1}}) & \rightarrow \pi^{n}\mc{O}_{X_{n+1}}.\\
\pi^{n}A & \mapsto \pi^{n}\mr{tr}(A)
\end{align*} 
\item We define the map 
\begin{align*}
1+ \pi^{n}\mr{tr}: 1 + \pi^{n}M(r, \mc{O}_{X_{n+1}})&\rightarrow 1 + \pi^{n}\mc{O}_{X_{n+1}}. \\
   1+ \pi^{n}A &\mapsto 1 + \pi^{n}\mr{tr}(A)
\end{align*}
\end{enumerate}
\end{defi}

Then we have the natural short exact sequence:

\begin{equation} \label{se1}
1 \to 1 + \mr{ker}(\pi^{n}\mr{tr}) \to 1 + \pi^{n}M(r, \mc{O}_{X_{n+1}}) \xrightarrow{1+ \pi^{n}\mr{tr}} 1 + \pi^{n}\mc{O}_{X_{n+1}} \to 1 .
\end{equation}

\vspace{0.2 cm}

\begin{lem}
Consider the natural surjective morphism 
\[\mr{SL}(r, \mc{O}_{X_{n+1}}) \xrightarrow{\alpha} \mr{SL}(r,\mc{O}_{X_{n}}).\]
Then $1 + \mr{ker}(\pi^{n}\mr{tr}) = \mr{ker}(\alpha)$. 
\end{lem}
 
\begin{proof} Note that the morphism $\alpha$ is a group homomorphism because it is induced by the ring homomorphism $\mc{O}_{X_{n+1}} \to \mc{O}_{X_n}$.
Let $N := (a_{ij})$ be a matrix in $\mr{SL}(r, \mc{O}_{X_{n+1}})$ with image the identity matrix in $\mr{SL}(r, \mr{O}_{X_{n}})$. 
Then $a_{ij} = \pi^{n}b_{ij}$ for $i \neq j$ with $b_{ij} \in \mc{O}_{X_{n+1}}$ and $a_{ii}  = 1+\pi^{n}b_{ii}$.  
Since $\pi^{n + 1} = 0$ in $\mc{O}_{X_{n+1}}$ and $(a_{ij}) \in \mr{SL}(r, \mc{O}_{X_{n+1}})$,

\[1 = \mr{det}(a_{ij}) = 1 + \pi^{n}\sum(b_{ii}).\]

\noindent Hence $\pi^{n}\sum(b_{ii})$ must be $0$.

\vspace{0.2 cm}
\noindent Since $\mr{tr}(N-\mr{Id})=\pi^n\sum_i b_{ii}$, this implies $N \in 1+\ker(\pi^n \mr{tr})$.
Hence $\mr{ker}(\alpha) \subseteq 1+\ker(\pi^n \mr{tr})$.
The reverse inclusion $1+\ker(\pi^n \mr{tr}) \subseteq \mr{ker}(\alpha)$ is direct.
\end{proof}

This gives us the following short exact sequence

\begin{equation}\label{se2}
 1 \to 1+ \mr{ker}(\pi^{n}\mr{tr}) \to \mr{SL}(r, \mc{O}_{X_{n+1}}) \to \mr{SL}(r,\mc{O}_{X_{n}}) \to 1.  
\end{equation}

\vspace{0.2 cm}

\begin{defi}
Recall the determinant map $\mr{det}: \mr{GL}(r, \mc{O}_{X_{n+1}}) \to \mc{O}^{\times}_{X_{n+1}}$ given by taking the determinant of the matrices.
Since  $\mr{SL}(r, \mc{O}_{X_{n+1}})$ are the matrices with determinant $1$, we have the following  exact sequence

\begin{equation}\label{se3}
1 \to  \mr{SL}(r, \mc{O}_{X_{n+1}}) \to \mr{GL}(r,\mc{O}_{X_{n+1}}) \xrightarrow{\mr{det}} \mc{O}_{ X_{n+1}}^{\times} \to 1.
\end{equation}
\end {defi}

\vspace{0.2 cm}

Using sequences (\ref{se1}), (\ref{se2}), (\ref{se3}) we obtain the following diagram  

  \[\begin{diagram}\label{d1}
        1 &  \rInto & 1 + \mr{ker}(\pi^{n}\mr{tr}) & \rTo & \mr{SL}(r, \mc{O}_{X_{n+1}}) & \rOnto & \mr{SL}(r, \mc{O}_{X_{n}}) & \rTo & 1\\
  & &\dInto &\circlearrowright&\dInto &\circlearrowright&\dInto & & \\
       1 &  \rInto & 1 + \pi^{n}M(r, \mc{O}_{X_{n+1}}) & \rTo & \mr{GL}(r, \mc{O}_{X_{n+1}}) & \rOnto & \mr{GL}(r, \mc{O}_{X_{n}}) & \rTo & 1 \\
 & &\dOnto^{1 + \pi^{n}\mr{tr}} &\circlearrowright&\dTo^{\mr{det}} &\circlearrowright&\dTo^{\mr{det}} & & \\
     1 &  \rInto & 1 + \pi^{n}\mc{O}_{X_{n+1}}  & \rTo & \mc{O}_{X_{n+1}}^{\times} & \rTo & \mc{O}_{X_{n}}^{\times} & \rOnto & 1 \\
   \end{diagram}.\]

\vspace{0.2 cm}
\noindent Note that the short exact sequence (\ref{se1}) splits i.e.  there exists a section $\phi$ to the trace map, 
\[\phi: \mc{O}_{X_{n+1}}(U) \to  M(r, \mc{O}_{X_{n+1}})(U), \ \  \lambda \mapsto  \begin{pmatrix}
  \lambda & 0 &\cdots & 0 \\
  0 & 0 & \cdots & 0\\
  \vdots  & \vdots  & \ddots & \vdots  \\
  0 & 0 & \cdots & 0
 \end{pmatrix}
  \]
\noindent for all $U \subseteq X_{n+1}$. 
Since the short exact sequence (\ref{se1}) is split exact,  we have the following short exact sequence:
 \[ 1 \to H^{1}(1 + \mr{ker}(\pi^{n}\mr{tr})) \to H^{1}(1 + \pi^{n}M(r, \mc{O}_{X_{n+1}})) \to  H^{1}(1 + \pi^{n}\mc{O}_{X_{n+1}}) \to 1. \]

\noindent Similarly, the short exact sequence (\ref{se3}) splits i.e there exists a section $\psi$ to the determinant map given by 

\[\psi: \mc{O}^{\times}_{X_{n+1}}(U) \to  \mr{GL}(r, \mc{O}_{X_{n+1}})(U), \ \ \  \lambda \mapsto  \begin{pmatrix}
  \lambda & 0 &\cdots & 0 \\
  0 & 1 & \cdots & 0\\
  \vdots  & \vdots  & \ddots & \vdots  \\
  0 & 0 & \cdots & 1
 \end{pmatrix}
  \]
  
\noindent for all for all $U \subseteq X_{n+1}$. 

\vspace{0.2 cm}

The following diagram summarises all this. 
 
 \begin{equation}\label{d2}
  \begin{diagram}
      H^{1}(1 + \mr{ker}(\pi^{n}\mr{tr})) & \rTo & H^{1}(\mr{SL}(r, \mc{O}_{X_{n+1}})) & \rTo & H^{1}(\mr{SL}(r, \mc{O}_{X_{n}})) \\
  \dTo & \circlearrowright &\dTo &\circlearrowright & \dTo\\
       H^{1}(1 + \pi^{n}M(r, \mc{O}_{X_{n+1}})) & \rTo & H^{1}(\mr{GL}(r, \mc{O}_{X_{n+1}})) & \rOnto & H^{1}(\mr{GL}(r, \mc{O}_{X_{n}})) \\
\dOnto_{1+ \pi^{n}\mr{tr}}&\circlearrowright &\dTo_{\mr{det}} &\circlearrowright &\dTo_{\mr{det}} \\
     H^{1}(1 + \pi^{n}\mc{O}_{X_{n+1}})  & \rTo^{\psi_1} & H^{1}(\mc{O}_{X_{n+1}}^{\times}) & \rTo^{\psi_2} & H^{1}(\mc{O}_{X_{n}}^{\times}) 
   \end{diagram}
 \end{equation}
 
\vspace{0.2 cm} 
\noindent Here the north-east square is a diagram of pointed sets, all the other groups are abelian. Note that, $1 + \mr{ker}(\pi^{n}\mr{tr})$ is a sheaf of abelian groups. Hence by Grothendieck vanishing, $H^2(1 + \mr{ker}(\pi^{n}\mr{tr}))=0$, implying surjectivity of the middle right arrow. 

We can now prove the following: 

\vspace{0.2 cm}
\begin{thm}\label{noobsg}
Suppose there exists a locally free sheaf $\mc{F}_n$ on $X_n$ with determinant $\mc{L}_{n} := \mc{L}_{\hat{R}}\otimes_{\hat{R}}R_{n}$ which is an extension of $\mc{F}_k$.
Then there exists an extension $\mc{F}_{n+1}$ of the locally free sheaf $\mc{F}_{n}$ such that $\mr{det}(\mc{F}_{n+1}) \simeq \mc{L}_{n+1}$ where $\mc{L}_{n+1} := \mc{L}_{\hat{R}}\otimes_{\hat{R}}R_{n+1}$.
\end{thm}

\begin{proof}
Since $X_k$ is a curve, by Grothendieck's vanishing theorem,
\[H^{2}(\Hc_{X_{k}}(\mc{F}_{k},\mc{F}_{k})\otimes_k \mf{m}^{n+1}/\mf{m}^{n+2}) = 0 \] 
\noindent for all $n\geq0$.
Hence by {\cite[Theorem $7.3$]{HD}}, there is no obstruction to extending $\mc{F}_{n}$ to a coherent sheaf say, $\mc{F}'_{n+1}$ over $X_{n+1}$. 
Furthermore, the sheaf $\mc{F}'_{n+1}$ is in fact a locally free sheaf on $X_{n+1}$ (see {\cite[Exercise $7.1$]{HD}}).
Let $\mc{L}'_{n+1}:= \mr{det}(\mc{F}'_{n+1})$.
If $\mc{L}'_{n+1}\simeq \mc{L}_{n+1}$, then we are done.
Suppose not, we now show how we can modify the extension $\mc{F}'_{n+1}$ so that its determinant bundle is in fact isomorphic to $\mc{L}_{n+1}$.

\vspace{0.2 cm} 

By {\cite[Theorem $7.3$]{HD}} the set of extensions of $\mc{L}_{n}$ on $X_{n+1}$ is a torsor under the action of $H^{1}(\mc{O}_{X_{n+1}}^{\times})$.
Hence there exists $\gamma \in H^{1}(1+ \pi^{n}\mc{O}_{X_{n+1}})$ such that  $[\mc{L}_{n+1}] = \gamma \bullet  [\mc{L}'_{n+1}]$ , where $\bullet$ indicates the torsor action of $H^{1}(1+ \pi^{n}\mc{O}_{X_{n+1}})$.
Since the morphism $1+ \pi^{n}\mr{tr}$ is surjective, there exists a preimage of $\gamma$, say $\Gamma \in H^{1}(1+ \pi^{n}M(r , \mc{O}_{X_{n+1}}))$.
Denote by $\mc{F}_{n+1} := \Gamma \circ \mc{F}'_{n+1}$ in $H^{1}(\mr{GL}(r, \mc{O}_{X_{n+1}}))$ where $\circ$ indicates the torsor action in  $H^{1}(\mr{GL}(r, \mc{O}_{X_{n+1}}))$.
Since the torsor action is compatible with taking determinant, the commutativity of the lower left square of diagram (\ref{d2}) implies $\mr{det}(\mc{F}_{n+1}) = [\mc{L}_{n+1}]$.
\end{proof}

\begin{prop}\label{p5} 
There exists a locally free sheaf $\mc{F}_{\hat{R}}$ on $X_{\hat{R}}$ such that $\mc{F}_{k}\simeq \mc{F}_{\hat{R}}\otimes_{\hat{R}}k$ with determinant $\mc{L}_{\hat{R}}$. 
\end{prop}

\begin{proof}
By {\cite[Ch$2$, Proposition $9.6$]{H}}, we obtain a locally free sheaf $\hat{\mc{F}}$ on the formal scheme $\mc{X}$. 
Then by {\cite[Theorem $8.4.2$]{FGA}}, there exists a locally free sheaf $\mc{F}_{\hat{R}}$ on $X_{\hat{R}}$ such that for the flat morphism $i:\mc{X} \to X_{\hat{R}}$, $\hat{\mc{F}}$ is isomorphic to $i^*(\mc{F}_{\hat{R}})$. 
By the commutativity of the following diagram
 \[\begin{diagram}
    \mc{X}&\rTo^{i}&X_{\hat{R}}\\
    \uTo^{u}&\ruTo^f\\
    X_k
   \end{diagram}\]
   
\noindent and {\cite[Ch$2$, Proposition $9.6$]{H}}, we obtain $f^*(\mc{F}_{\hat{R}})=(i \circ u)^{*}\mc{F}_{\hat{R}}$ is isomorphic to $u^*\hat{\mc{F}}=\mc{F}_k$.
\end{proof}

Now we obtain a Gieseker semistable locally free sheaf $\mc{F}_{R}$ of rank $r$ on $X_R$ with determinant $\mr{det}(\mc{F}_R)\simeq \mc{L}_R$. 
We do this using the locally free sheaf $\mc{F}_{\hat{R}}$ on $X_{\hat{R}}$ obtained in Proposition \ref{p5} and Artin approximation. We refer the reader to \cite{artalg} for all the relevant definitions and results on Artin approximation.

\vspace{0.2 cm}
The following is a consequence of the Quot functor being locally of finite presentation.  

\begin{lem}\label{quotfp}
Let $f:X_R \to \msp (R)$ be a flat, projective morphism and $\mc{H}$ a free sheaf on $X_R$ of the form
$\oplus_{i=1}^N \mo_{X_R}$ for some $N$. 
Recall the Quot-functor $\mc{Q}uot_{X_R/\msp (R)/\mc{H}}^P$ for a fixed Hilbert polynomial $P$. 
Given a projective system $\{Z_i\}_{i \in I}$ of affine schemes,
the natural morphism 
\[\rho:\varinjlim\limits_{i \in I} \mc{Q}uot_{X_R/\msp (R)/\mc{H}}^P (Z_i) \to \mc{Q}uot_{X_R/\msp (R)/\mc{H}}^P (\varprojlim\limits_{i \in I} Z_i)\] is bijective.
\end{lem}

\begin{proof}
By \cite[Proposition $4.4.1$]{Ser}, the Quot-functor $\mc{Q}uot_{X_R/\msp (R)/\mc{H}}^P$ is represented by a projective $\msp(R)$-scheme. 
In particular, the natural morphism 
\[\phi: \mr{Quot}_{X_R/\msp (R)/\mc{H}}^P \to \msp(R)\] is of finite type. 
Since $\msp(R)$ is locally noetherian, the morphism $\phi$ is locally of finite presentation.
Then by \cite[Proposition $8.14.2$]{ega43} the Quot-functor $\mc{Q}uot_{X_R/\msp (R)/\mc{H}}^P$ is locally of finite presentation.
Hence the lemma follows.
\end{proof}

\vspace{0.2 cm}
We use the following lemma to prove Theorem \ref{foundfk}.

\begin{lem}\label{in07}
Consider the fiber product:
 \[\begin{diagram}
    X_{\hat{R}}&\rTo^{j_1}&X_R\\
    \dTo^{\hat{f}}&\square&\dTo^f\\
    \Spec (\hat{R})&\rTo^{j_0}&\msp (R)
   \end{diagram}\]
where $j_0$ is the natural morphism.
Denote by $i_n:X_n \hookrightarrow X_R$ the natural closed immersion.

\noindent If $\mathcal{F}_{R}$ is an invertible sheaf on $X_R$ satisfying $i_n^*{\mc{F}_R} \cong \mo_{X_n}$ for all $n \ge 1$, then $\mc{F}_R \cong \mo_{X_R}$.
\end{lem}

\begin{proof}
By {\cite[Theorem $1.12$]{artalg}} the natural morphism \[i:\mr{Pic}(X_R) \to \varprojlim \mr{Pic}(X_n)\] is injective. Since $i(\mc{F}_R)=i(\mo_{X_R})$, we have $\mc{F}_R \cong \mo_{X_R}$.
\end{proof}

\vspace{0.2 cm}
Now we apply Artin approximation to our situation to obtain a Gieseker semistable locally free sheaf with determinant $\mc{L}_R$ on the surface $X_R$.

\begin{thm}\label{foundfk}
Let $\mc{F}_{\hat{R}}$ be the locally free sheaf on $X_{\hat{R}}$ with determinant $\mc{L}_{\hat{R}}:= j_1^*\mc{L}_R$ obtained using Proposition \ref{p5}. 
Then, there exists a geometrically stable locally free sheaf $\mc{F}_{R}$ on $X_R$ with determinant $\mc{L}_R$.
\end{thm}

\begin{proof}
By Lemma \ref{quotfp}, the Quot functor is locally of finite presentation.  
Using Artin's approximation theorem \cite[Theorem $1.12$]{artalg}, 
we conclude that there exists a coherent sheaf $\mc{F}_{R}$ on $X_R$ such that  

\begin{equation}\label{eqn7}
 i_n^*\mc{F}_{R} \cong {i'_n}^*\mc{F}_{\hat{R}}, \ \ \ \forall  \ \ n \geq 1.
\end{equation}

\noindent  where $i'_n:X_n \to X_{\hat{R}}$ is the morphism induced by the natural morphism $\hat{R} \to R/m^n$.
By Theorem \ref{existvbdettl}, $\mc{F}_{k}\simeq \mc{F}_{R}\otimes_{R}k$. 
Since locally freeness and geometric semistability are open properties.
Therefore $\mc{F}_{R}$ is locally free and geometrically stable. 

\vspace{0.2 cm}
Let $\mc{L}:=\det(\mc{F}_{R})$.
Using the fact that determinant commutes with pull-back and the isomorphism (\ref{eqn7})
\[i_n^*\mc{L} \cong \mr{det} i_n^*\mc{F}_{R} \cong \mr{det} {i'_n}^*\mc{F}_{\hat{R}}.\]
\noindent By assumption $\det(\mc{F}_{\hat{R}}) \cong j_1^*\mc{L}_R$. 
Hence $\mr{det}{i'_n}^*\mc{F}_{\hat{R}}\cong i_n'^* \circ j_1^* \mc{L}_R$.
By the universal property of inverse limits $j_1 \circ i_n'=i_n$, hence 
$i_n'^* \circ j_1^* \mc{L}_R \cong i_n^*\mc{L}_R$.
Therefore for all $n \geq 1$, 
\[i_n^*(\mc{L} \otimes_{\mo_{X_R}} \mc{L}_R^{\vee}) \cong \mo_{X_n}.\] 
\noindent Hence by Lemma \ref{in07}, $\mc{L} \cong \mc{L}_R$ i.e., $\det(\mc{F}_{R}) \simeq \mc{L}_R$.
This proves the proposition.
\end{proof}

\end{document}